\documentclass[11pt,a4paper]{amsart}

\usepackage[margin=25truemm,top=30truemm,bottom=30truemm]{geometry}
\usepackage{graphicx}
\usepackage{enumerate}
\usepackage{amsthm}
\usepackage{amssymb}
\usepackage{tikz-cd}
\usepackage{here}
\usepackage{mathtools}
\usepackage{amssymb}
\usepackage{xcolor}

\newtheorem{Prop}{Proposition}[section]
\newtheorem{Thm}[Prop]{Theorem}

\newtheorem{Rem}[Prop]{Remark}

\theoremstyle{definition}
\newtheorem{Def}[Prop]{Definition}
\newtheorem{Ex}[Prop]{Example}

\makeatletter
\@addtoreset{equation}{section}
\makeatother

\title{Two-step nilpotent Lie algebras obtained by quivers and geometric structures}
\author{Fumika Mizoguchi} 

\date{}
\thanks{This work was supported by JST SPRING, Grant Number JPMJSP2139. This work was partly supported by MEXT Promotion of Distinctive Joint Research Center Program JPMXP0723833165.}

\address[F.~Mizoguchi]{Department of Mathematics, Osaka Metropolitan University, Osaka City, Japan 558-8585} 
\email{sw23882u@st.omu.ac.jp}

\DeclareMathOperator{\g}{\mathfrak{g}}
\DeclareMathOperator{\n}{\mathfrak{n}}
\DeclareMathOperator{\Ric}{\mathrm{Ric}}

\DeclareMathOperator{\Path}{\mathrm{Path}}

\begin{document}
	\maketitle
	\begin{abstract}
		In the study of geometric structures on nilmanifolds, two-step nilpotent Lie algebras obtained by graphs play an important role. 
Recently, additional examples of nilpotent Lie algebras have been constructed from finite quivers without cycles. 
These latter examples can have arbitrarily high degrees of nilpotency, and admit Riemannian Ricci soliton metrics. 
In this paper, we study nilpotent Lie algebras obtained by finite quivers without cycles that are two-step nilpotent, and we prove that they can also be obtained by graphs. 
Using this relationship, we demonstrate that every two-step nilpotent Lie algebra obtained by a finite quiver without cycles admits a pseudo-Riemannian Ricci-flat metric.
Additionally, we also classify these nilpotent Lie algebras that admit symplectic structures.
	\end{abstract}
	\section{Introduction}
		In geometry, it is an important problem to study whether a given manifold admits special geometric structures. In the context of Lie groups, it is natural to inquire about the existence of left-invariant geometric structures. Namely, it is a natural problem to ask whether a given Lie group admits special left-invariant geometric structures. In general,  the classification of Lie groups with special left-invariant geometric structures is difficult. 
		
		We are interested in left-invariant geometric structures on nilpotent Lie groups.  Examples include special Riemannian metrics such as Einstein or Ricci soliton, special pseudo-Riemannian metrics, and symplectic structures. However, even in the case of two-step nilpotent Lie groups, these structures are not completely understood. 
		
		In the study of geometry of two-step nilpotent Lie groups, one typical example is given by the class of two-step nilpotent Lie algebras obtained by simple directed graphs, as studied by Dani--Mainkar ([DM]), see also Section~\ref{relationship}. Additional examples of nilpotent Lie algebras have been constructed from finite quivers without cycles (\cite{MT}). This method can construct nilpotent Lie algebras with arbitrarily high nilpotency steps, and these Lie algebras admit Ricci solitons (\cite{MT}). In this paper, we focus on the two-step case and prove the following theorem. 
		\begin{Thm}
			All two-step nilpotent Lie algebras obtained by finite quivers without cycles can also be obtained by simple directed graphs. 
		\end{Thm}
		
		Moreover, by applying previous researches on two-step nilpotent Lie algebras obtained by graphs, we investigate whether two-step nilpotent Lie algebras obtained by finite quivers without cycles admit certain special geometric structures. Using the theorem by Conti--del Barco--Rossi (\cite{CDR}), we prove the following theorem. 
		\begin{Thm}
			All two-step nilpotent Lie algebras obtained by finite quivers without cycles admit pseudo-Riemannian Ricci-flat metrics. 
		\end{Thm} 
		Next, we consider symplectic structures. According to the research by Pouseele--Tirao (\cite{PT}), nilpotent Lie algebras obtained by graphs do not always admit symplectic structures. Furthermore, they posed the condition whether these Lie algebras admit symplectic structures or not. Applying their result, we prove the following theorem. 
		\begin{Thm}
			A two-step nilpotent Lie algebra obtained by a finite quiver without cycles admits a symplectic structure if and only if the dimensions of the Lie algebra is even and each connected component in the quiver is one of the $(2, 2)$, $(1, n)$, $(n, 1)$ and $(1, 0)$-type quivers.  
		\end{Thm}
		We refer to Section~\ref{relationship} for the definition of the $(m, n)$-type quiver. In this paper, we focus only on two-step nilpotent Lie algebras obtained from finite quivers without cycles. In this context, we prove that all of them admit pseudo-Riemannian Ricci-flat metrics, and classify those that admit symplectic structures. While higher-step cases remain unclear, we are interested in studying them further. The results presented in this paper will be useful for future research. 
	\section{Preliminaries}
			In this section, referring to \cite{MT}, we recall the method of constructing Lie algebras obtained by quivers. In particular, from finite quivers without cycles, one obtains finite-dimensional nilpotent Lie algebras.  We recall some basic notions on quivers. A quiver is a directed graph where loops and multiple arrows between two vertices are allowed. The precise definition is given as follows. 
			\begin{Def}
				A {\it quiver} $Q=(V, E, s, t)$ is a quadruple consisting of two sets $V$, $E$, and two maps $s, t: E\to V$. Elements of $V$ and $E$ are called {\it vertices} and {\it arrows}, respectively. For an arrow $\alpha\in E$, the vertices $s(\alpha)$ and $t(\alpha)$ are called the {\it source} and the {\it target} of $\alpha$, respectively.
			\end{Def}
			Quivers consist of vertices and arrows. Note that the notion of paths is also important, which is defined as follows. 
			\begin{Def}
				Let $Q=(V, E, s, t)$ be a quiver. A {\it path of length $m$} is a sequence $\alpha_1\alpha_2\cdots \alpha_m$, where $\alpha_1, \ldots, \alpha_m\in E$ and they satisfy $t(\alpha_i)=s(\alpha_{i+1})$ for $i=1,\ldots, m-1$.
			\end{Def}
			\begin{Ex}\label{Ex_quiver}
				The following gives a quiver: $V=\{v_1,v_2,v_3,v_4,v_5\},E=\{a,b,c,d,e\}$, $s(a)=v_1$, $s(b)=v_2$, $s(c)=v_3$, $s(d)=v_3$, $s(e)=v_4$, $t(a)=v_2$, $t(b)=v_4$, $t(c)=v_4$, $t(d)=v_4$, $t(e)=v_5$. It is illustrated as in Figure \ref{quiver}. Paths of this quiver are $a$, $b$, $c$, $d$, $e$, $ab$, $be$, $ce$, $de$, and $abe$. 
			\end{Ex}
			\begin{figure}[H]
				\begin{center}
					\includegraphics[scale=0.8]{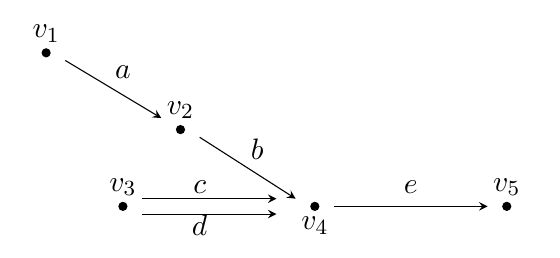}
					\caption{The quiver in Example \ref{Ex_quiver}}
					\label{quiver}
				\end{center}
			\end{figure}
			For a path $x=\alpha_1\alpha_2\cdots \alpha_m$ in a quiver $Q$, the vertices $s(\alpha_1)$ and $t(\alpha_m)$ are also called the {\it source} and the {\it target} of $x$, and denoted by $s(x)$ and $t(x)$ respectively.  In this way, the maps $s, t: E\to V$ can be extended to $s, t: \Path(Q)\to V$, where $\Path(Q)$ is the set of all paths in $Q$. A path whose source and target coincide is called a {\it cycle}. A quiver is said to be {\it finite} if the set of vertices and the set of arrows are both finite. 
			\begin{Def}[\cite{MT}]
				Let $Q=(V, E, s, t)$ be a quiver. The {\it path algebra} $\n_Q$ is defined as an $\mathbb{R}$-vector space having $\Path(Q)$ as basis with the following product: it is bilinear, and the product between two paths $x$ and $y$ is defined by the concatenation, that is 
				\[
            				x\cdot y=
            				\begin{cases}
            					xy & (\text{if }xy \text{ is a path}), \\
            					0 & (\text{otherwise}).
            				\end{cases}
            			\]
				The space $\n_Q$ equipped with the bracket product $[x, y]\coloneqq x\cdot y-y\cdot x$ is called the {\it Lie algebra obtained by a quiver}. 
			\end{Def}
			For the path algebras $(\n_Q, \cdot)$, we refer to \cite{ASS}. Note that path algebras usually contain vertices as paths of length 0. However, we have to note that the Lie algebras $(\n_Q, [,])$ do not contain vertices, that is, we only consider paths of length $\geq1$. We refer to \cite{MT} for examples of nilpotent Lie algebras obtained by quivers. 
			\begin{Prop}[\cite{MT}]
				Let $Q$ be a finite quiver without cycles. Then, there exists the maximum of the lengths of paths, which is called the {\it length of a quiver}. Moreover, if $Q$ is a quiver of length $m$ then the Lie algebra $\n_Q$ is an $m$-step nilpotent Lie algebra. 
			\end{Prop}
			Therefore, to study the case that $\n_Q$ is of two-step, we should consider quivers of length $2$. 			
	\section{A relationship between nilpotent Lie algebras obtained by quivers and by graphs}\label{relationship}
			In this section, we introduce a relationship between nilpotent Lie algebras obtained by quivers of length $2$ and graphs. First, we recall nilpotent Lie algebras obtained by graphs. A directed simple graph is an ordered pair $\Gamma=(V, E)$ consisting of a set of vertices $V$, and a set of edges $E\subset\{(x, y)\in V\times V\mid x\neq y\}$. 
			For $e=(x, y)\in E$, the vertex $x$ is called the {\it source} of $e$ and $y$ is called the {\it target} of $e$. 
			\begin{Def}[\cite{DM}]
				Let $\Gamma=(V, E)$ be a simple directed graph. Let $\n_\Gamma$ be the $\mathbb{R}$-vector space having the set $V \cup E$ as basis. The space $\n_\Gamma$ becomes a two-step nilpotent Lie algebra with Lie bracket given by 
				\[
                        			[x,y]=
                        			\begin{cases}
                        				e & (x,y\in V,e=(x,y)), \\
                        				-e & (x,y\in V,e=(y,x)), \\
                        				0 & (\textrm{otherwise}).
                        			\end{cases}
            			\]	
			\end{Def}
			Second, we prove that for each two-step nilpotent Lie algebra $\n_Q$ obtained by a quiver $Q$, there exists a simple directed graph $\Gamma$ such that $\n_Q$ is isomorphic to $\n_\Gamma$. 
			Note that the following quiver in Figure~\ref{quiver_2step} is called the {\it $(m, n)$-type quiver}. 
			\begin{figure}[H]
				\begin{center}
					\includegraphics{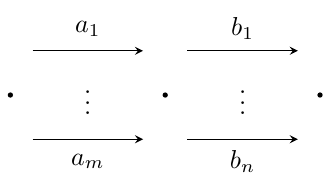}
					\caption{The $(m, n)$-type quiver}
					\label{quiver_2step}
				\end{center}
			\end{figure}
			\begin{Prop}\label{2step}
				Let $Q=(V, E, s, t)$ be a quiver of length $2$. Then, there exist some $(m_i, n_i)$-type quivers $Q_i$~$(1\leq i\leq r)$ and an abelian Lie algebra $\mathfrak{a}$ such that $\n_Q\simeq(\oplus_{i=1}^r\n_{Q_i})\oplus\mathfrak{a}$. 
			\end{Prop}			
			\begin{proof}
				Let us consider the following subset of $V$, which is a finite set: 
				\[
					V'\coloneqq\{v\in V\mid \exists a, b\in E \text{ s.t. } t(a)=s(b)=v\}=\{v_1, \ldots, v_r\}. 
				\]
				First, for each $i\in\{1, \ldots, r\}$, we consider a quiver $Q_i=(V_i, E_i, s_i, t_i)$ defined as follows: 
				\begin{align*}
					E_{i, 1}&\coloneqq\{a\in E\mid t(a)=v_i\}, \\
					E_{i, 2}&\coloneqq\{b\in E\mid s(b)=v_i\}, \\
					E_i&\coloneqq E_{i,1}\cup E_{i, 2}. 
				\end{align*}
				A set $V_i$ and maps $s_i$, $t_i$ are naturally determined from $Q$, but strictly speaking, they are defined as follows: 
				\begin{align*}
					V_i&\coloneqq\{v_i, w_{i,1}, w_{i, 2}\}, \\
					s_i(x)&\coloneqq
					\begin{cases}
						w_{i, 1} & (x\in E_{i, 1}), \\
						v_i & (x\in E_{i, 2}), 
					\end{cases} \\
					t_i(x)&\coloneqq
					\begin{cases}
						v_i & (x\in E_{i, 1}), \\
						w_{i, 2} & (x\in E_{i, 2}). 
					\end{cases}
				\end{align*}						
				Moreover, we put $m_i=\#E_{i,1}$ and $n_i=\#E_{i, 2}$. Then, each quiver $Q_i$ is of $(m_i, n_i)$-type. Note that any $E_i$ and $E_j$ are disjoint if $i\neq j$, by the definition of $V'$. We put $E_0\coloneqq E\setminus \bigcup_{i=1}^r E_i$. Then, the set $E$ can be decomposed into disjoint subsets $E_0$, $E_1, \ldots, E_r$. 
				
				Second, we define an abelian Lie algebra $\mathfrak{a}$ by 
					\[
						\mathfrak{a}\coloneqq\mathbb{R}^{\#E_0}. 
					\]
					
				Finally, we prove that it satisfies $\n_Q\simeq(\oplus_{i=1}^r\n_{Q_i})\oplus\mathfrak{a}$. The set $\Path(Q)$ is decomposed as follows: 
				\begin{align*}
					\Path(Q)&=\textstyle{E_0\sqcup\left(\bigsqcup_{i=1}^r(E_{i,1}\cup E_{i, 2})\right)\sqcup\left(\bigsqcup_{i=1}^r \{ab\in\Path(Q)\mid a\in E_{i, 1}, b\in E_{i, 2}\}\right)} \\
					&=\textstyle{E_0\sqcup\left(\bigsqcup_{i=1}^r\Path(Q_i)\right)}. 
				\end{align*}
				Therefore, $\n_Q$ is isomorphic to $(\oplus_{i=1}^r\n_{Q_i})\oplus\mathfrak{a}$ as a vector space. It remains to show that these spaces are isomorphic as a Lie algebra, that is, $[\mathrm{span}~E_0, \Path(Q)]=0$ and $[\Path(Q_i), \Path(Q_j)]=0$ if $i\neq j$. 
				
				Take $\alpha\in E_0$ and $x\in\Path(Q)$. Assume that $[\alpha, x]\neq 0$, that is $t(\alpha)=s(x)$ or $t(x)=s(\alpha)$. In the case of $t(\alpha)=s(x)$, it satisfies $t(\alpha)=s(x)\in V'$ by the definition of $V'$. Therefore, $\alpha\in \bigcup_{i=1}^rE_i$, which is a contradiction. One can show the case of $t(x)=s(\alpha)$ similarly. 
				
				Next, take $x\in\Path(Q_i)$ and $y\in\Path(Q_j)$. Assume that $i\neq j$ and $[x, y]\neq0$, that is $t(x)=s(y)$ or $t(y)=s(x)$. If $x$ or $y$ are paths of length $2$, $[x, y]=0$ holds, since $Q$ is a quiver of length $2$. Therefore, $x\in E_i$ and $y\in E_j$. In the case of $t(x)=s(y)$, it holds $y\in E_i$ by $t(x)=s(y)=v_i$. This is a contradiction, since $E_i$ and $E_j$ are disjoint if $i\neq j$. One can show the case of $t(y)=s(x)$ similarly, which completes the proof. 
			\end{proof}
			Let $Q$ be the $(m, n)$-type quiver and $\n_Q$ be the two-step nilpotent Lie algebra obtained by $Q$. Then, we would like to obtain $\n_Q$ from a graph. A {\it complete bipartite graph} is a special graph where every vertex of the first set is connected to every vertex of the second set. Namely, it is a graph of the form $\Gamma=(V_1\sqcup V_2, E)$ with $E=V_1\times V_2$ as the following graph in Figure~\ref{graph}. The directed graph $\Gamma$ is called the {\it directed complete $(m, n)$ bipartite graph} if $\#V_1=m$ and $\#V_2=n$. 
			\begin{figure}[H]
				\begin{center}
					\includegraphics[scale=0.8]{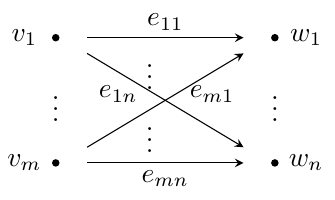}
					\caption{The directed complete $(m, n)$ bipartite graph}
					\label{graph}
				\end{center}
			\end{figure}
			\begin{Prop}\label{mn-type}
				Let $Q$ be the $(m, n)$-type quiver and $\Gamma=(V, E)$ be the directed complete $(m, n)$ bipartite graph. Then, it satisfies $\n_Q\simeq\n_\Gamma$. In particular, it satisfies $\#V=m+n$ and $\#E=mn$. 
			\end{Prop} 
			\begin{proof}
				The two-step nilpotent Lie algebra $\n_Q$ is given by
				\[
					\n_Q=\mathrm{span}\{a_i, b_j, a_ib_j\mid 1\leq i\leq m, 1\leq j\leq n\}, 
				\]
				where the Lie brackets are given by $[a_i, b_j]=a_ib_j$. Also, the two-step nilpotent Lie algebra $\n_\Gamma$ is
				\[
					\n_\Gamma=\mathrm{span}\{v_i, w_j, e_{ij}=(v_i, w_j)\mid  1\leq i\leq m, 1\leq j\leq n\}
				\]
				where the Lie brackets are given by $[v_i, w_j]=e_{ij}$. Therefore, it is easy to see $\n_Q\simeq\n_\Gamma$ as Lie algebras. 
			\end{proof}
			\begin{Rem}\label{(1,0)}
				Note that the following quiver in Figure~\ref{(1,0)-type} is called the {\it $(1, 0)$-type quiver}. The Lie algebra obtained by this quiver is abelian. Then, for the $(1, 0)$-type quiver, Propositions~\ref{mn-type} also hold. Namely, the nilpotent Lie algebra obtained by the $(1, 0)$-type quiver is isomorphic to the nilpotent Lie algebra obtained by the graph $\Gamma=(\{v\}, \phi)$. In Proposition~\ref{2step}, an abelian Lie algebra $\mathfrak{a}$ is obtained by the union of $(1, 0)$-type quivers. Therefore, any nilpotent Lie algebra obtained by a quiver of length $2$ is isomorphic to the direct sum of nilpotent Lie algebras obtained by $(m_i, n_i)$-type quivers, where $m_i\geq1$ and $n_i\geq0$. 			
			\end{Rem}
			\begin{figure}[H]
				\begin{center}
					\includegraphics[scale=0.7]{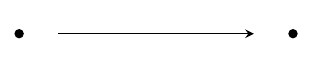}
					\caption{$(1, 0)$-type}
					\label{(1,0)-type}
				\end{center}
			\end{figure}
			\begin{Thm}\label{quiver_graph}
				Let $Q$ be a quiver of length $2$. Then, $\n_Q$ is also obtained by a graph, that is, there exists a directed simple graph $\Gamma$ such that $\n_Q\simeq\n_\Gamma$.
			\end{Thm}
			\begin{proof}
				The nilpotent Lie algebra $\n_Q$ obtained by $Q$ is isomorphic to the direct sum of $\n_{Q_i}$, where $Q_i$ is the $(m_i, n_i)$-type quiver. Note that it could be $(m_i, n_i)=(1, 0)$ (see Remark~\ref{(1,0)}). Also, for each $Q_i$, there exists  a complete $(m_i, n_i)$ bipartite graph $\Gamma_i=(V_i, E_i)$ such that $\n_{Q_i}\simeq\n_{\Gamma_i}$. It holds $\n_Q\simeq\n_\Gamma$, where $\Gamma\coloneqq\bigcup_i \Gamma_i$. 
			\end{proof}
	\section{Pseudo-Riemannian  Ricci-flat metrics}
		In this section, we prove that  two-step nilpotent Lie algebras obtained by quivers always admits pseudo-Riemannian Ricci-flat metrics. Let $G$ be a simply-connected Lie group, $g$ be a left-invariant pseudo-Riemannian metric and $(\g, \langle, \rangle)$ be the Lie algebra with indefinite inner product corresponding to $(G, g)$. Curvatures of $(G, g)$ can be determined by the corresponding $(\g, \langle, \rangle)$. We state our result on metric Lie algebras $(\g, \langle, \rangle)$. We refer to \cite{CR} the definition of the Ricci curvature $\Ric$ with respect to a pseudo-Riemannian metircs

		\begin{Def}
			Let $\g$ be a Lie algebra, $\langle, \rangle$ an indefinite inner product on $\g$, and $\{e_i\}$ be a pseudo-orthonormal basis of $\g$, where $\langle e_i, e_i\rangle=\pm1$. Then, the pair $(p, q)$ is called the {\it signature} of $\langle, \rangle$ if $p=\#\{i\mid \langle e_i, e_i\rangle=1\}$ and $q=\#\{i\mid \langle e_i, e_i\rangle=-1\}$. 
		\end{Def}
		\begin{Def}
			Let $(\g, \langle, \rangle)$ be a Lie algebra with an indefinite inner product  and $\Ric: \g\to\g$ be the Ricci curvature. The metric $\langle, \rangle$ is said to be {\it pseudo-Riemannian Ricci-flat} if $\Ric=0$. 
		\end{Def}
		For pseudo-Riemannian Ricci-flat metrics on nilpotent Lie algebras obtained by graphs, the following theorem is known.  
		\begin{Thm}[\cite{CDR}]\label{CDR}
			Any two-step nilpotent Lie algebra obtained by a graph admits a pseudo-Riemannian Ricci-flat metric. 
		\end{Thm}
		By applying this theorem, for two-step nilpotent Lie algebras obtained by finite quivers without cycles, we obtain the following theorem. 
		\begin{Thm}
			Let $Q$ be a quiver of length $2$ and $\n_Q$ be the two-step nilpotent Lie algebra obtained by $Q$. Then, $\n_Q$ always admits a pseudo-Riemannian Ricci-flat metric.
		\end{Thm}
		\begin{proof}
			There exists a directed simple graph $\Gamma$ such that $\n_Q\simeq \n_\Gamma$ by Theorem~\ref{quiver_graph}. The Lie algebra $\n_\Gamma$ admits a pseudo-Riemannian Ricci-flat metric by Theorem \ref{CDR}. Therefore, $\n_Q$ admits a such metric.  
		\end{proof}
		\begin{Rem}
			Let $Q$ be the $(m, n)$-type quiver. Then, the signature $(p, q)$ of the pseudo-Riemannian Ricci-flat metric obtained in this theorem satisfies 
			\[
				p, q\geq m+n-1. 
			\]
		\end{Rem}
	\section{Symplectic structures}\label{symplectic}
		In this section, we prove the condition for two-step nilpotent Lie algebras obtained by quivers to admit a symplectic structure. Let $G$ be a simply-connected Lie group with dimension $2n$ and $\g$ be its corresponding Lie algebra. A left-invariant $2$-form $\omega$ is called a {\it symplectic structure} if $d\omega=0$ and $\omega^n\neq0$. We define a symplectic structure on $\g$ and state our result on $\g$. 
		\begin{Def}
			Let $g$ be a Lie algebra. A $2$-form $\omega$ is called a {\it symplectic structure} on $\g$ if it satisfies $d\omega=0$ and $\omega^n\neq 0$. 
		\end{Def}
		For symplectic structures on nilpotent Lie algebras obtained by graphs, the following theorem is known. 
		\begin{Thm}[\cite{PT}]\label{PT}
			Let $\Gamma=(V, E)$ be a directed simple graph and $\n_\Gamma$ be the two-step nilpotent Lie algebra obtained by $\Gamma$. Then, $\n_\Gamma$ admits a symplectic structure if and only if the dimension of $\n_\Gamma$ is even, and in each connected component of $\Gamma$, the number of edges does not exceed the number of vertices. 
		\end{Thm}
		By applying this theorem, for two-step nilpotent Lie algebras obtained by finite quivers without cycles, we obtain the following theorem. 
		\begin{Thm}
			Let $Q$ be a quiver of length $2$ and $\n_Q$ be a two-step nilpotent Lie algebra obtained by $Q$.  Then, $\n_Q$ admits a symplectic structure if and only if the following conditions hold: 
			\begin{enumerate}[(i)]
				\item the dimension of $\n_Q$ is even, 
				\item each connected component in $Q$ is one of the $(2, 2)$, $(1, n)$, $(n, 1)$ and $(1, 0)$-type quivers as follows. 
			\end{enumerate}
		\end{Thm}
		\noindent
		\begin{minipage}[t]{.49\linewidth}
			\begin{figure}[H]
				\begin{center}
					\includegraphics[scale=0.7]{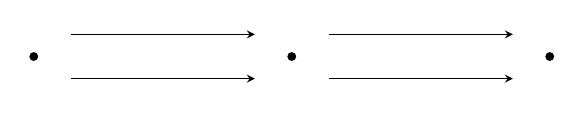}
					\caption{$(2, 2)$-type}
					\label{symplectic1}
				\end{center}
			\end{figure}
		\end{minipage}
		\begin{minipage}[t]{.49\linewidth}
			\begin{figure}[H]
				\begin{center}
					\includegraphics[scale=0.7]{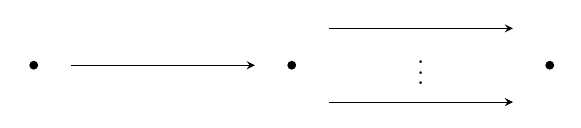}
					\caption{$(1, n)$-type}
					\label{symplectic2}
				\end{center}
			\end{figure}
		\end{minipage}\\
		\begin{minipage}[t]{.49\linewidth}
			\begin{figure}[H]
				\begin{center}
					\includegraphics[scale=0.7]{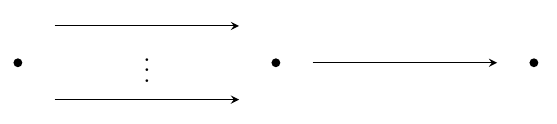}
					\caption{$(n, 1)$-type}
					\label{symplectic3}
				\end{center}
			\end{figure}
		\end{minipage}
		\begin{minipage}[t]{.49\linewidth}
			\begin{figure}[H]
				\begin{center}
					\includegraphics[scale=0.7]{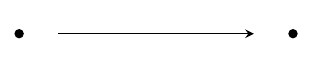}
					\caption{$(1, 0)$-type}
					\label{symplectic4}
				\end{center}
			\end{figure}
		\end{minipage}
		\begin{proof}
			There exists a directed simple graph $\Gamma=(V, E)$ such that $\n_Q\simeq \n_\Gamma$ by Theorem~\ref{quiver_graph}. Therefore, the dimension of $\n_Q$ is even if and only if the dimension of $\n_\Gamma$ is even. By Theorem \ref{PT}, it is enough to show that each connected component $\Gamma_i=(V_i, E_i)$ of $\Gamma$ satisfies $\#E_i\leq\#V_i$ if and only if each connected quiver $Q_i$ in $Q$ is one of the above. By Proposition~\ref{2step}, each connected component of $Q$ is isomorphic to $(m_i, n_i)$-type quiver $Q_i$. Moreover, by Proposition~\ref{mn-type}, there exists a complete $(m_i, n_i)$ bipartite graph $\Gamma_i=(V_i, E_i)$ such that $\n_{Q_i}\simeq\n_{\Gamma_i}$. In particular, $\#V_i=m_i+n_i$ and $\#E_i=m_in_i$. 
			In the case of $m_i\neq0$ and $n_i\neq0$, we have
			\[
				\#E_i\leq\#V_i\Leftrightarrow m_in_i\leq m_i+n_i\Leftrightarrow 1\leq \frac{1}{m_i}+\frac{1}{n_i}. 
			\]
			This inequality holds if and only if $(m_i, n_i)=(2, 2), (1, n), (n, 1)$. In the case of $(m_i, n_i)=(1, 0)$, it satisfies $\#E_i\leq\#V_i$. This complete the proof. 			
		\end{proof}		
		The two-step nilpotent Lie algebras obtained by the following quivers in Figures~\ref{symplectic_example} and \ref{symplectic_example2} are 14-dimensional and 12-dimensional, respectively. Therefore, both are even dimensional. However, neither satisfies the condition for connected components. Therefore, they do not have symplectic structures.   
		\begin{figure}[H]
			\begin{center}
				\includegraphics[scale=0.8]{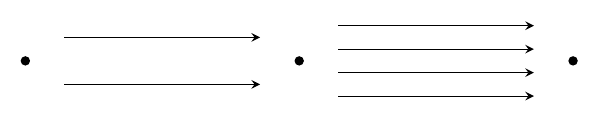}
				\caption{The $(2, 4)$-type quiver}
				\label{symplectic_example}
			\end{center}
		\end{figure}
		\begin{figure}[H]
			\begin{center}
				\includegraphics[scale=0.8]{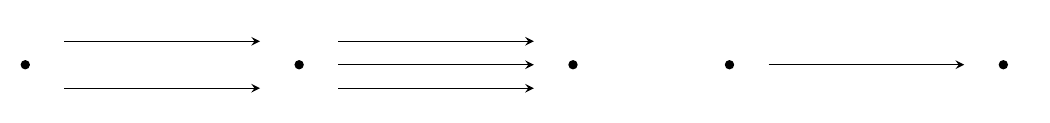}
				\caption{The disjoint union of the $(2, 3)$-type quiver and the $(1, 0)$-type quiver}
				\label{symplectic_example2}
			\end{center}
		\end{figure}
	\section*{Acknowledgements}
		The author wishes to extend her heartfelt thanks to Professors Ali Baklouti and Hideyuki Ishi for organizing the 7th Tunisian-Japanese Conference on "Geometric and Harmonic Analysis on Homogeneous Spaces and Applications." The author is also grateful to Hiroshi Tamaru for his valuable discussions.

\end{document}